\documentclass[12pt]{amsart}
\usepackage{amsmath,amssymb,amsthm,xspace,a4wide,amscd,eucal}
\usepackage[all]{xy}
\usepackage{latexsym}
\usepackage{textcomp}
 \usepackage[T1]{fontenc}
\usepackage[colorlinks=true]{hyperref}
\usepackage{xspace}
\numberwithin{equation}{section}

\newcommand{\propref}[1]{Proposition~\ref{#1}}

\newcommand{\eqnref}[1]{~(\ref{#1})}

\newtheorem{thm}{Theorem}[section]
\newtheorem{lem}[thm]{Lemma}

\newtheorem{prop}[thm]{Proposition}

\newtheorem{defi}[thm]{Definition}

\newtheorem{definition}[thm]{Definition}
\newtheorem{remark}[thm]{Remark}
\newtheorem{theorem}[thm]{Theorem}
\newtheorem{lemma}[thm]{Lemma}
\newtheorem{proposition}[thm]{Proposition}
\newtheorem{corollary}[thm]{Corollary}

\newcommand{\C}{\ensuremath{\mathbb{C}}\xspace}

\newcommand{\Z}{\ensuremath{\mathbb{Z}}\xspace}

\newcommand{\ad}{\operatorname{ad}\xspace}

\renewcommand{\phi}{\varphi}
\def\D{\Delta}

\def\cD{\mathcal D}

\begin{document}
\title{Localization of free field realizations of affine Lie algebras}
\author{Vyacheslav Futorny, Dimitar Grantcharov, Renato A. Martins }

\address{\noindent
Instituto de Matem\'atica e Estat\'istica, Universidade de S\~ao
Paulo,  S\~ao Paulo SP, Brasil} \email{futorny@ime.usp.br,
renatoam@ime.usp.br}
\address{\noindent
University of Texas at Arlington,  Arlington, TX 76019, U.S.} \email{grandim@uta.edu}

\date{}

\begin{abstract}
We  use localization technique to construct new families of
irreducible modules of affine Kac-Moody algebras. In particular,
localization is applied to the first  free field realization of the affine
Lie algebra $A_1^{(1)}$ or, equivalently, to imaginary Verma modules.
\end{abstract}
\maketitle

\section{Introduction}

Free field realizations of affine Kac-Moody algebras play important role in representation theory and conformal field theory. Free field realizations of the affine Lie  algebra $A_1^{(1)}$ appeared first in \cite{JK} for zero central charge and in \cite{W} for arbitrary central charge. Using geometric methods, these constructions were generalized in \cite{BF} for arbitrary Borel subalgebras, and in \cite{FF1}, \cite{FF2} for higher rank affine Lie algebras. 

Among the Borel subalgebras of  $A_1^{(1)}$, there is one of special interest - the so called {\em natural} Borel subalgebra. The Verma modules corresponding to the natural Borel subalgebra are usually referred as {\em imaginary} Verma modules. Contrary to the classical Verma modules, the  imaginary Verma modules have both finite and infinite weight multiplicities and much more complicated structure (\cite{BF}, \cite{C1}, \cite{F1}, \cite{FS}, \cite{JK}). The free field realizations of imaginary Verma modules are normally referred as {\em first free field realizations} of $A_{1}^{(1)}$. A natural way to construct such realizations for an affine Lie algebra ${\mathfrak g}$ is to consider an embedding of  $\mathfrak g$ to an algebra of differential operators of the corresponding semi-infinite flag manifold $G/B_{\rm nat}$, see \cite{FF1}. In the case $\mathfrak g = A_1^{(1)}$ this leads to a homomorphism $U(A_1^{(1)}) \to {\mathcal A} ({\bf x}, {\bf y})$, where ${\mathcal A} ({\bf x}, {\bf y})$ is the  algebra of formal power series of differential operators on ${\mathbb C} [{\bf x}, {\bf y}]$, and  ${\mathbb C} [{\bf x}, {\bf y}]$ is the ring of polynomials of $x_i$, $i \in {\mathbb Z}$, $y_j$, $j >0$.

In this paper we apply localization technique to the first free field realization of $A_1^{(1)}$ and obtain a new family of free field realizations. The twisted localization functor provides  new examples
of irreducible weight modules that are dense (i.e. with maximal sets of weights) and have  infinite weight multiplicities. By applying a parabolic induction to such modules we can construct new families of irreducible dense modules 
with infinite weight multiplicities for any affine Lie algebra. These families contribute to the classification problem of irreducible weight modules of affine Lie algebras, a problem which is still largely open. Note that the  problem of classifying the irreducible weight modules with  finite weight  multiplicities was completely solved in \cite{FT} in the case of a nonzero central charge, and in \cite{DG} in the case of a zero central charge.    
Examples of weight modules with infinite-dimensional weight spaces have been studied also in \cite{CM}, and although the modules in \cite{CM} are over  the Virasoro algebra, they are similar in nature to the ones obtained in the present paper.

Localization with respect to an Ore subset can be applied to the second free field realization (Wakimoto modules), which will be treated in a subsequent publication.

The present paper  can also be considered as a first step towards the study of the weight modules of  $A_{1}^{(1)}$ coming from weight modules of ${\mathcal A} ({\bf x}, {\bf y})$. The simple weight modules of the algebra ${\mathcal A}^{\rm poly} ({\bf x}, {\bf y})$ of polynomial differential operators of ${\mathbb C} [{\bf x}, {\bf y}]$ were classified in \cite{FGM}. These modules are modules of ${\mathcal A} ({\bf x}, {\bf y})$ and can be realized as quotients of twisted localizations of the defining module ${\mathbb C} [{\bf x}, {\bf y}]$. In the present paper we address the question when  a simple weight ${\mathcal A}^{\rm poly} ({\bf x}, {\bf y})$-module coming from one $x_i$-localization and from finitely many $y_j$-localizations of  ${\mathbb C} [{\bf x}, {\bf y}]$ remains simple when considered as an $A_{1}^{(1)}$-module. 

The functor mapping a module $M$ to the quotient 
$\mathcal D(M)/M$, where ${\mathcal D}$ is the localization functor corresponding to a real root, plays a central role in our paper. It is interesting to note that the functor $ M \mapsto {\mathcal D} (M)/M$ can be considered as a variation of an Arkhipov functor. For finite-dimensional semi-simple Lie algebra, the latter is a composition of the former with a twist of an element of the Weyl group, \cite{Ark1}. Properties of these functors were studied also in \cite{E}. Generalized versions of these functors for affine Lie algebras and for quantum groups were introduced in \cite{Ark2} and \cite{CE1, CE2}, respectively.

The organization of the paper is as follows. In Section 3 we collect some preliminary results on the first free field realization and twisted localization of weight $A_{1}^{(1)}$-modules. Our main results for the first free field realization modules are formulated in Section 4, while the proofs are provided in Section 5. The statements and the proofs of the main results in terms of imaginary Verma modules are presented in Section 6.

\section{Notation and conventions}

The ground field will be ${\mathbb C}$. We let $\mathbb N$ denote the set of positive integers. For a Lie algebra ${\mathfrak a}$ by $U(\mathfrak a)$ we denote the universal enveloping algebra of ${\mathfrak a}$.

We fix $(\, \, , \, )$ to  be the Killing form of the Lie algebra $\mathfrak{sl}_2$. Consider the affine Lie algebra
$$\mathfrak{g}=\widehat{\mathfrak{sl}}_2=\left( \mathfrak{sl}_2\otimes
\mathbb C [t, t^{-1}] \right) \oplus \mathbb C c\oplus \mathbb C d,$$ where
$c$ is the central element, and $d$ is the degree derivation. The
commutation relations on $\mathfrak{g}$ are
$$
[a \otimes t^m, b \otimes t^n] = [a,b] \otimes t^{m+n} + \delta_{m, -n} m \, (a,b) c, \quad
[d, a \otimes t^m] = m a \otimes t^m,  \quad [c, {\mathfrak g}] =0,
$$
where $a, b \in {\mathfrak g}$, $m, n \in \Z$, and $\delta_{i,j}$ is Kronecker's delta. 

We fix a positive root $\beta$  and a standard basis $e,f,h$  of  $\mathfrak{sl}_2$.

Then
$$e_n=e\otimes t^n, \, h_n=h\otimes t^n, \,
f_n=f\otimes t^n, n\in \Z,$$ together with $c$ and $d$,  form a basis of
$\widehat{\mathfrak{sl}}_2$. Furthermore, $\mathfrak{h}=\mathbb C h_0 \oplus \mathbb C
c\oplus \mathbb C d$ is a Cartan subalgebra of
$\widehat{\mathfrak{sl}}_2$.  Denote by $\Delta $ the root system of
$\widehat{\mathfrak{sl}}_2$ corresponding to $(\widehat{\mathfrak{sl}}_2, \mathfrak{h})$. Let $\delta \in {\mathfrak h}^*$ be defined by $\delta |_{\mathbb C h_0 \oplus \mathbb C
c} = 0$ and $\delta(d) = 1$.  Then $\Delta= \D^{\mathsf {re}} \cup \D^{\mathsf {im}}$, where  $\D^{\mathsf {im}} =  \{n \delta \, | \,  n \in \Z, n \neq 0\}$ and $\Delta^{\mathsf{re}} =  \{ \pm\beta + n\delta\ |\ n \in \Z\}$.

We will  use  ${\bf x}$ to denote the set $\{ x_n \: | \; n \in \Z\}$ and ${\bf y}$ for the set $\{ y_m \: | \; m \in {\mathbb N}\}$. More generally, for a subset $I$ of $\Z$, ${\bf x}_I:=\{ x_n \: | \; n \in I\}$. Similarly, we define ${\bf y}_S$ for $S \subset \mathbb N$. Furthermore, for $I \subset \Z$ and $S \subset \mathbb N$ we set $\widehat{\bf x}_I:= \{ x_n \: | \; n \in \Z, n \notin I\}$ and $\widehat{\bf y}_S:= \{ y_m \: | \; m \in {\mathbb N}, m \notin S\}$. We write $\widehat{\bf x}_i$ for $\widehat{\bf x}_{\{ i\}}$ and $\widehat{\bf y}_j$ for $\widehat{\bf y}_{\{ j\}}$.

Let $\C [{\bf x}, {\bf y}] = \C [x_n, y_m \, | \,  n \in \Z, m \in \mathbb N]$ (polynomials in ${\bf x}$ and ${\bf y}$) and $\C [{\bf x}^{\pm1}, {\bf y}^{\pm1}] = \C [x_n^{\pm1}, y_m^{\pm1} \, | \,  n \in \Z, m \in \mathbb N]$ (Laurent polynomials in ${\bf x}$ and ${\bf y}$). More generally, we write $\C [{\bf x}_I, {\bf x}_{I'}^{\pm1},  {\bf y}_S, {\bf y}_{S'}^{\pm1}] $ for the ring of polynomials in ${\bf x}_I, {\bf y}_S$ and Laurent polynomials in ${\bf x}_{I'}, {\bf y}_{S'}$. For convenience we will also use the expressions $\widehat{\bf x}_I^{\pm1}$ and $\widehat{\bf y}_S^{\pm1}$ when needed. For example $\C [x_i, \widehat{\bf x}_i^{\pm1}, {\bf y}^{\pm1}] = \C [x_i, {\bf x}_{{\mathbb Z} \setminus \{ i\}}^{\pm1}, {\bf y}^{\pm1}] $.

Denote by ${\mathcal A} ({\bf x}, {\bf y})$ the algebra of formal power series of  differential operators on $\C [{\bf x}, {\bf y}] $. Namely, the elements of ${\mathcal A} ({\bf x}, {\bf y})$ are $\sum_{I,J} P_{I,J} \frac{\partial}{\partial {\bf x}_I}\frac{\partial}{\partial {\bf y}_J}$ where for $I=(i_1,...,i_n)$ and $J = (j_1,...,j_m)$,  $\frac{\partial}{\partial {\bf x}_I}=\frac{{\partial}^n}{\partial x_{i_1} \cdots \partial x_{i_n}}$,$\frac{\partial}{\partial {\bf y}_J} =  \frac{{\partial}^m}{\partial y_{j_1} \cdots \partial y_{j_m}}$, and $P_{I,J} \in \C [{\bf x}, {\bf y}]$. We will write  $\partial x_i$ and $\partial y_j$ for $\frac{\partial}{\partial x_i}$ and $\frac{\partial}{\partial y_j}$, respectively.

\section{Preliminaries} \label{sec-prel}

A $\mathfrak{g}$-module $V$ is called a {\em weight module} if $V=\bigoplus_{\lambda\in \mathfrak{h}^*}V_{\lambda}$, where 
$V_{\lambda}=\{v\in V| zv=\lambda(z)v, \forall z\in \mathfrak{h}\}$. The set of all  $\lambda\in \mathfrak{h}^*$ for which 
$V_{\lambda}\neq 0$ is called the {\em support} of $V$. If $V$ is irreducible then its support is a subset of a fixed coset of 
$\mathfrak{h}^*/Q$ where $Q = {\Z} \beta + {\Z} \delta$ is the root lattice. In the case when the support of $V$ is  a full coset $\lambda + Q$ of  $\mathfrak{h}^*/Q$, we will call $V$ a {\em dense module}. A module on which all real root elements $e_n$ and $f_n$ act  injectively will be called {\em torsion free}. Clearly, every torsion free module is dense, but the converse is not necessarily true. For counterexamples see Theorems \ref{main5} and \ref{main6}.

\subsection{Imaginary Verma Modules}
Throughout the paper we fix the   triangular decomposition $\mathfrak g = {\mathfrak g}_{-} \oplus {\mathfrak h}\oplus {\mathfrak g}_+$ of $\mathfrak g$, where
$${\mathfrak g}_{\pm}= \left( \bigoplus_{n\in \mathbb Z}{\mathfrak g}_{\pm \beta +n\delta}\right) \oplus \left( \bigoplus _{m\in \mathbb N}{\mathfrak g}_{\pm m\delta}\right) .$$ 

Let $\lambda \in \mathfrak h^*$.  We endow $\mathbb C$ with a $U({\mathfrak g}_{+} \oplus { \mathfrak h})$-module structure by  setting
$(a+b)\cdot 1 = \lambda(b) 1$, $ a\in {\mathfrak g}_{+}, b \in {\mathfrak h}$, and denote the corresponding module by $\mathbb C_{\lambda}$. The induced module
$
M(\lambda) = U(\mathfrak g) \otimes_{U({\mathfrak g}_{+} \oplus { \mathfrak h})}\mathbb C_{\lambda} 
$
is called the {\it imaginary Verma module}  of $\mathfrak g$ of highest weight $\lambda$.  The following summarize  some important properties of the imaginary Verma modules (see Proposition 1 and Theorem 1 in \cite{F1}).
\begin{prop}[\cite{F1}]\label{prop-imag}
Let $\lambda \in \frak h^*$. Then $M(\lambda)$ has the following properties.
\begin{enumerate}
\item[(i)]  $M(\lambda)$ has a unique maximal submodule.
\item[(ii)]   We have $\dim M(\lambda)_{\lambda} = 1$ and $0< \dim M(\lambda)_{\lambda - k\delta} < \infty$  for any positive integer $k$.  If
$\mu \neq \lambda - k \delta$ for any integer $k \ge 0$  and
$\dim M(\lambda)_{\mu} \neq 0$  then $\dim M(\lambda)_{\mu} = \infty$.
\item[(iii)]   The module $M(\lambda)$ is irreducible if and only if
$\lambda(c) \neq 0$.
\item[(iv)]  Let $\lambda(c)=0$. Then $\mathfrak M=U(\mathfrak g)\left( \sum_{k\in \mathbb N}M_{\lambda - k \delta}\right) $ 
is a proper submodule of $M(\lambda)$. Moreover, $M(\lambda)/{\mathfrak M}$ is irreducible if and only if $\lambda(h_0)\neq 0$.
\end{enumerate}
\end{prop}

\subsection{First free field realization} \label{sec-fffr}
Consider the following map from $\widehat{\mathfrak{sl}}_2$ to the algebra ${\mathcal A} ({\bf x}, {\bf y})$.\\
$\begin{aligned} f_n\mapsto  x_n,
\end{aligned}$\\
$\begin{aligned}
h_n\mapsto -2\sum_{m\in \Z}x_{m+n}\partial x_{m}
+\delta_{n<0}y_{-n} + \delta_{n>0}2nK\partial y_{n}
+\delta_{n,0}J,
\end{aligned}$\\
$\begin{aligned}e_n\mapsto -\sum_{m,k\in \Z}x_{k+m+n}\partial x_k\partial x_{m}
+\sum_{k>0}y_{k}\partial x_{-k-n}+ 2K\sum_{m>0}m\partial
y_m\partial x_{m-n} +(Kn+J)\partial x_{-n},
\end{aligned}$\\
$\begin{aligned}  d \mapsto \displaystyle\sum_{i\in\mathbb Z}ix_i\partial x_i-\displaystyle\sum_{j >0}jy_j\partial y_j, \;  c \mapsto K.
\end{aligned}$\\

\begin{prop} \label{FFFRhom}
The above map  gives rise to a homomorphism $\Phi: U (\widehat{\mathfrak{sl}}_2) \mapsto {\mathcal A} ({\bf x}, {\bf y})$ of associative algebras.
\end{prop}
\begin{proof}
The fact that $U (\mathfrak{sl}_2 \otimes {\mathbb C}[t^{\pm1}] ) \mapsto {\mathcal A} ({\bf x}, {\bf y}) $ is a homomorphism is well-known (see for example \cite{FB}, \cite{JK}). It remains to check $\Phi([d,g_n])=[\Phi(d),\Phi(g_n)]$, for $g_n=e_n,f_n,h_n$, which is relatively easy. \end{proof}

We call the homomorphism $\Phi$ the {\it  first free field realization of $\widehat{\mathfrak{sl}}_2$} (FFFR for short). With the aid of $\Phi$ we endow any ${\mathcal A} ({\bf x}, {\bf y})$-module $M$ with an $\widehat{\mathfrak{sl}}_2$-module structure. In the particular case when $M$ is the defining ${\mathcal A} ({\bf x}, {\bf y})$-module $\C [{\bf x},{\bf y}]$ we obtain a realization of imaginary Verma modules. This was done in the case $K = 0$ in \cite{JK} and later for arbitrary $K$ in \cite{BF}.  
 This construction has been generalized for the affine Lie algebras $\widehat{\mathfrak{sl}}_n$ in \cite{C2}. For reader's convenience we state the result for $\widehat{\mathfrak{sl}}_2$.
 
\begin{theorem}[\cite{BF}, \cite{JK}] \label{th-im-ver} The $\widehat{\mathfrak{sl}}_2$-module $\C [{\bf x},{\bf y}]$ is isomorphic to the imaginary Verma module $M(\lambda)$, where $\lambda \in {\mathfrak h}^*$ is defined by $\lambda (h_0) = J$, $\lambda (c) = K$, $\lambda (d) = 0$. 
\end{theorem} 
 
\begin{remark}  \label{rem-d}
\begin{enumerate}
\item[(i)] Note that one  can easily obtain a free field realization of $M(\lambda)$ in the case when $\lambda(d) = D$ is nonzero by shifting the map for $d$: $d \mapsto  D+ \displaystyle\sum_{i\in\mathbb Z}ix_i\partial x_i-\displaystyle\sum_{j >0}jy_j\partial y_j$.  
\item[(ii)] The free field realization in \cite{JK} is isomorphic to the quotient of $M(\lambda)$, $\lambda(c) =0$, by the submodule generated by $h_n \otimes 1$, $n<0$. Also, the realization in \cite{BF} is equivalent to our FFFR after applying two anti-involutions. See \cite[\S 3.3]{CF} for details.
\end{enumerate}
\end{remark} 
By Proposition \ref{prop-imag}(iii) the module $\C [{\bf x},{\bf y}]$ is irreducible if and only if $K\neq 0$. If $K=0$, by Proposition \ref{prop-imag}(iv), the quotient of  $\C [{\bf x},{\bf y}]$ by the module generated by $y_m, m>0$, is  irreducible if and only if $J\neq 0$. This quotient is isomorphic to
 $\C[\bf{x}]$ and has $ \widehat{\mathfrak{sl}}_2$-module structure  through the homomorphism $\Phi' : U ( \widehat{\mathfrak{sl}}_2) \to  {\mathcal A} ({\bf x})$ defined by:\\
\begin{eqnarray*} 
 f_n & \mapsto&   x_n,\\
h_n & \mapsto & -2\sum_{m\in \Z}x_{m+n}\partial x_{m}+\delta_{n,0}J,\\
e_n & \mapsto &  -\sum_{m,k\in \Z}x_{k+m+n}\partial x_k\partial x_{m}+J\partial x_{-n},\\
d & \mapsto &  \displaystyle\sum_{i\in\mathbb Z}ix_i\partial x_i, \;  c \mapsto 0.
\end{eqnarray*}
The analog of Theorem \ref{th-im-ver} for the defining module $\C [{\bf x}]$ of ${\mathcal A} ({\bf x})$ is the following. 

\begin{theorem} \label{th-im-ver-quot} The  ${\mathcal A} ({\bf x})$-module $\C [{\bf x}]$ considered as an $\widehat{\mathfrak{sl}}_2$-module through $\Phi'$ is isomorphic to $M'(\lambda)$, where $\lambda \in {\mathfrak h}^*$ is defined by $\lambda (h_0) = J$, $\lambda (c) = 0$, $\lambda (d) = 0$. 
\end{theorem} 
The $\widehat{\mathfrak{sl}}_2$-module  $\C[\bf{x}]$ is irreducible if and only if $J \neq 0$. If $J=0$ then the module $\C[\bf{x}]$ has unique irreducible quotient isomorphic to $\C$. 

\subsection{Localization of weight $\widehat{\mathfrak{sl}}_2$-modules} \label{loc}
In this subsection we recall the definition of the localization functor on weight modules. For details we refer the reader to \cite{D} and \cite{M}.

We set $U:=U(\widehat{\mathfrak{sl}}_2)$.
For every $g= e_i$ or $g = f_i$, $ i \in \Z$, the multiplicative set
${\bf S}_g:=\{ g^n \; | \; n \in \Z_{\geq 0} \} \subset U$
satisfies Ore's localizability conditions
because $\ad g$ acts locally finitely on $U$. Let
$\cD_g U$ be the localization of $U$ relative to ${\bf S}_{g}$.
For every weight module $M$ we denote by $\cD_g M$ the {\it
$g$--localization of $M$}, defined as $\cD_g M =
\cD_g U \otimes_U M$. If $g$ acts  injectively on $M$, then $M$ can be naturally viewed as 
a submodule of $\cD_g M$, and $\cD_g^2 M = \cD_g M$. Furthermore, if
$g$ is injective on $M$, then it is bijective on $M$ if and
only if $\cD_g M = M$.

\subsection{Generalized conjugations.} \label{gencon}
 For $z \in \C$ and $u \in \cD_g U$  we set
\begin{equation} \label{theta}
\Theta_z(u):= \sum_{j \geq 0} \binom{z}{j}\,
( \ad g)^j (u) \, g^{-j},
\end{equation}
where $\binom{z}{j}= \frac{z( z-1)\cdots (z-j+1)}{j!}$. Since $\ad g$ is locally nilpotent on 
$\cD_g U$, the sum above is
actually finite. Note that for $z \in \Z$ we have $\Theta_z(u) =
g^z u g^{-z}$.  For a $\cD_g U$-module $M$ by
$\Psi_g^z M$ we denote the $\cD_g U$-module $M$ twisted by
the action
$$
u \cdot v^z := ( \Theta_z (u)\cdot v)^z,
$$
where $u \in \cD_g U$, $v \in M$, and $v^z$ stands for the
element $v$ considered as an element of $\Psi_g^z M$. In
particular, if $g = e_i$ (respectively, $g = f_i$), then $v^z \in M^{\lambda + z(\beta+i\delta)}$ (respectively, $v^z \in M^{\lambda + z(- \beta+i\delta)}$ ) whenever $v \in
M^\lambda$. 

In the case ${z} \in {\mathbb Z}$, there is a natural isomorphism of ${\cD}_{g} {\mathcal U}$-modules $M \to \Psi_{F}^{\bf z} M$ given by $m \mapsto ({g}^{z} \cdot m)^{z}$ with inverse map defined by $n^{z} \mapsto {g}^{-z} \cdot n$.
In view of this isomorphism, for ${z} \in {\mathbb Z}$, we will identify $M$ with  $\Psi_{g}^{z}M$, and for any ${z} \in {\mathbb C}$ will write ${g}^{z} \cdot m$ (or simply ${g}^{z} m$) for $m^{-z}$ whenever $m \in M$.

The following lemma is straightforward.

\begin{lem} \label{lem-conj}
Let $M$ be a $\cD_{g} U$-module, $v \in M$, $u \in \cD_g U$ and $z,w \in \C$. Then
\begin{itemize}
\item[(i)] $\Psi_g^z (
\Psi_g^w M) \simeq \Psi_g^{z+w}M $, in particular $g^z \cdot (g^w \cdot v) =
g^{z+w} \cdot v$.

\item[(ii)] $g^z \cdot (u \cdot (g^{-z}
\cdot v)) = \Theta_z(u) \cdot v$.
\end{itemize}
\end{lem}

\medskip

In what follows we set $\cD_g^z M:=\Psi_g^z (\cD_g M)$ and refer to it as a 
{\it twisted localization of $M$}. Note that the localization and the twisted localization functors are exact.

\begin{remark}
If $\alpha = \beta+ i \delta$, and $g = f_i$, then the functors $\cD_g$ and $\cD_g^z$ are often denoted by  $\cD_{\alpha}$ and $\cD_{\alpha}^z$, respectively. The $\alpha$-notation for the localization functors is convenient when higher rank (affine) Lie algebras are considered, see \cite{DG}, \cite{Gr}.
\end{remark}

The following is straightforward.
\begin{lemma} \label{sup-lemma} Let $M$ be a $U$-module, $z \in {\mathbb C}$, and $\alpha = - \beta + i \delta$. Then ${\rm Supp}\, \cD_{f_i}^z M = z\alpha + {\rm Supp}\, M + {\mathbb Z} \alpha$.
\end{lemma}
\begin{proof}
Let $N = \cD_{f_i} M$. Then  ${\rm Supp} \, N =  {\rm Supp}\, M + {\mathbb Z} \alpha$. The lemma follows  from the fact that  for $v^z \in N^{\lambda + z\alpha}$ if and only if  $v \in N^\lambda$. 
\end{proof}

\section{Irreducibility of localized modules} \label{sec-loc}

From now on we fix $i \in {\mathbb Z}$. 

Unless otherwise specified, every module $M$ of ${\mathcal A} ({\bf x}, {\bf y})$ is considered as a module of $\widehat{\mathfrak{sl}}_2$ via $\Phi$. For simplicity,  we will often write $M$ both for the ${\mathcal A} ({\bf x}, {\bf y})$-module and the corresponding $\widehat{\mathfrak{sl}}_2$-module.  Fix  a  finite subset $S=\{ k_1,\dots,k_n\} $ of $\mathbb N$ ($S = \emptyset$ is allowed). We first define four spaces of Laurent polynomials:
\begin{eqnarray*}
N_S & := & \mathbb C[{\bf x},{\bf y}_S^{{\pm1}},\widehat{\bf y}_S],\\
N_{i,S} & :=& \mathbb C[x_i^{\pm1},{\bf y}_S^{{\pm1}},\widehat{\bf x_i},\widehat{\bf y}_S], \\
I_S & := & \mathbb C[{\bf x},{\bf y}_{S\setminus\{k_1\}}^{\pm1},\widehat{\bf y}_{S\setminus\{k_1\}}]+\dots+\mathbb C[{\bf x},{\bf y}_{S\setminus\{k_n\}}^{\pm1},\widehat{\bf y}_{S\setminus\{k_n\}}], \\
 I_{i,S}& :=& \mathbb C[{\bf x},{\bf y}_{S}^{\pm1},\widehat{\bf y}_{S}]+\mathbb C[x_i^{ \pm 1},{\bf y}_{S\setminus\{k_1\}}^{\pm1},\widehat{{\bf x_i}},\widehat{\bf y}_{S\setminus\{k_1\}}]+\dots+\mathbb C[x_i^{\pm 1},{\bf y}_{S\setminus\{k_n\}}^{\pm1},\widehat{{\bf x_i}},\widehat{\bf y}_{S\setminus\{k_n\}}].
\end{eqnarray*}
We furthermore set 
$$
M_{S}:=N_S/I_S, \; M_{i,S}:=N_{i,S}/I_{i,S}.
$$

Using Proposition \ref{FFFRhom}  we can easily verify the following.

\begin{proposition} The homomorphism $\Phi$   defines a representation of  $\widehat{\mathfrak{sl}}_2$ on the spaces $N_{S}$, $N_{i,S}$, $M_S$, and $M_{i,S}$. We also have the following  vector space isomorphisms
\begin{equation*}
M_{S}\simeq 
y_{k_1}^{-1}\ldots y_{k_n}^{-1}\mathbb C[{\bf x},{\bf y}_{S}^{-1},\widehat{\bf y}_{S}], \; 
M_{i,S}\simeq x_i^{-1}y_{k_1}^{-1}\ldots y_{k_n}^{-1}\mathbb C[x_i^{-1},{\bf y}_{S}^{-1},\widehat{{\bf x_i}},\widehat{\bf y}_{S}].
\end{equation*}
\end{proposition}

With the aid of the  isomorphisms in the above proposition we will present every element in $M_{S}$ (respectively, $M_{i,S}$) uniquely as a polynomial with negative powers of $y_{k_1},...,y_{k_n}$ (respectively, of $y_{k_1},...,y_{k_n}$ and $x_i$).

Recall that if $K=0$ and $J \neq 0$ then the module $\C[{\bf x},{\bf y}]$ is reducible and has a unique irreducible quotient isomorphic to $\C[\bf{x}]$.  In fact, the quotient map $\C[{\bf x},{\bf y}] \to \mathbb C[\bf{x}]$ is nothing else but the evaluation homomorphism $ f({\bf x}, {\bf y})\mapsto f({\bf x}, {\bf 0})$.

In the case $K=0$, we define $N_i =  \mathbb C[x_i^{\pm 1}, \widehat{{\bf x_i}}]$, $I_i = \mathbb C[{\bf x}]$, and $M_i = N_i/I_i$. 
\begin{proposition} The homomorphism $\Phi'$   defines a representation of  $\widehat{\mathfrak{sl}}_2$ on the spaces $N_{i}$ and $M_{i}$. We also have the following  vector space isomorphism
$$M_i \simeq x_i^{-1}\C[x_i^{-1},\hat{\bf{x}}_i].$$
\end{proposition}
With the aid of the  isomorphisms in the above proposition we will present every element in $M_i$ as a polynomial with negative powers of $x_i$.

Our goal is to understand the  $\widehat{\mathfrak{sl}}_2$-module structure of  $M_{S}$,  $M_{i,S}$  and of $M_i$.  We will denote these modules also by  $\Phi(S)$,   $\Phi(i,S)$, and $\Phi'(i)$, respectively.  Our first main result of the paper is the following.

\begin{theorem}\label{main}  Let $i$ and $S$ be as above. 
\begin{itemize}
\item[(i)] Let $K \neq 0$. Then the representation $\Phi(S)$ is  irreducible, while  $\Phi(i,S)$  is irreducible if and only if $J-iK\notin {\mathbb Z}$.

\item[(ii)] Let $K = 0$. Then the representation  $\Phi'(i)$ is irreducible if and only if $J \notin {\mathbb Z}$.
\end{itemize}

\end{theorem}

Theorem~\ref{main} allows us to explicitly construct  new irreducible representations of  $\widehat{\mathfrak{sl}}_2$ applying the localization functors to imaginary 
Verma modules. These new modules are dense  weight modules with  infinite multiplicities as proved in Section \ref{sec-im-ver}.

\section{Proof of  Theorem \ref{main}}

As before  $U=U(\widehat{\mathfrak{sl}}_2)$. 
By abuse of notation we will write $u(v)$ instead of $\Phi(u)(v)$ for $u\in U$,  $v\in M_S$ or $v\in M_{i,S}$, as well as $u'(v')$  instead of $\Phi'(u')(v')$ for $u'\in U$,  $v'\in M_i$.

 We  first introduce some standard notation.

\begin{defi}\label{deg}
Let $v=x_i^{-p}y_{k_1}^{-\gamma_1}\dots y_{k_n}^{-\gamma_n}y_{j_1}^{\beta_1}\ldots y_{j_l}^{\beta_l}x_{i_1}^{\alpha_1}\ldots x_{i_k}^{\alpha_k}\in M_{i,S}$, where $p, \gamma_1,\ldots,\gamma_n,$ $\beta_1,\ldots,\beta_l,$ $\alpha_1,\ldots,\alpha_k\in\mathbb N$. Define
\begin{eqnarray*}
\deg_x(v)& =& -p+\alpha_1 + \cdots + \alpha_k,\\
\deg_{x^+}(v)& =& \alpha_1 + \cdots + \alpha_k,\\
\deg_y(v)& =& -\gamma_1-\dots-\gamma_n+\beta_1+\dots+\beta_l,\\
\deg_{y^+}(v_q)& =& \beta_1+\dots+\beta_l,
\end{eqnarray*}

The definitions above naturally extend to the case of an arbitrary
 $v$ in $M_{i,S}$. For example,  if $v=v_1+\dots+v_t$ where $v_1,\dots,v_t$ are monomials then we set
\begin{equation*}
\deg_y(v)=\max\{\deg_y(v_j)|j=1,\dots,t\},
\end{equation*}
and call $\deg_y(v)$ the {\em $y$-degree} of $v$. We call $v=v_1+\dots+v_t$ $\deg_y${\em-homogeneous} if the monomials $v_1,...,v_t$ have the same  $y$-degree. We similarly introduce $\deg_x (v)$ and 
$\deg_{x^+} (v)$ for arbitrary $v$ in $M_{i,S}$, as well as the notions {\em $x$-degree, $x^{+}$-degree, $x$-homogeneous}, and {\em $x^{+}$-homogeneous}. Finally, we introduce analogously the same notation and notions for the modules $M_S$ and $M_i$, respectively.
\end{defi}

\subsection{Primitive vectors in the case $J-iK \in {\mathbb Z}$.} In this subsection we focus on the case when $J-iK$ is an integer. The main goal is to prove that the modules $M_{i,S}$ and $M_i$ have $e_{-i}$-primitive vectors. As a corollary we will show that $M_{i,S}$ and $M_i$ are reducible. 

The following notation will be used in this subsection only.
$$
{\mathcal E}_i =\sum_{s_1>i,s_2>i}x_{s_1+s_2-i}\partial x_{s_1}\partial x_{s_2} = \sum_{s_1,s_2>0}x_{s_1+s_2+i}\partial x_{s_1+i}\partial x_{s_2+i}. 
$$

\begin{lemma} \label{lemma-e-n} Let $N>0$. Then there is $\deg_x$-homogeneous element $g$ in ${\mathbb C} [{\bf x}]$ with $\deg (g) = \deg_x (g) = N+1$ such that $\partial x_{j} (g)=0$ for $j\leq i$ and ${\mathcal E}_i^N (g) = 0$.
\end{lemma}
\begin{proof}
Let $i_1, \ldots, i_{N+1}$ be such that $i_s>i$ for every $s$. Then it is not hard to compute  
$$
{\mathcal E}_i^N( x_{i_1}\cdots x_{i_{N+1}} ) = c_N x_M, 
$$
where $c_N = 2^N \prod_{s=2}^{N+1}\binom{s}{2}$ and $M = i_1+\cdots +i_{N+1}$. Hence $g = x_{i_1}\cdots x_{i_{N+2}} - x_{i_1'}\cdots x_{i_{N+2}'}$, where $i_s'>i$, and $\sum_{s=1}^{n+2} i_s = \sum_{s=1}^{n+2} i_s'$, will satisfy the conditions of the lemma.
\end{proof}

\begin{proposition}\label{primitive} Let $J-iK \in {\mathbb Z}$.
\begin{itemize}
\item[(i)] If $K \neq 0$ then there is nonzero $v$ in  $M_{i, S}$ such that $e_{-i} (v) = 0$.

\item[(ii)] If  $K = 0$  then there is nonzero $v$ in  $M_{i}$ such that $e_{-i} (v) = 0$.
\end{itemize}

\end{proposition}
\begin{proof} First, for any $k$, we find a primitive vector $v$ of $M_{i, \emptyset}$ (i.e. for $S= \emptyset$) with the property $\partial y_j (v) = 0$ for every $j>0$. This, in particular, implies  (ii). For this first step we fix a nonnegative integer $N$ such that $N \geq J-iK - 2$. The strategy is to find two $e_{-i}$-primitives vectors in the module $N_{i,\emptyset}$, the second one of which will be nonzero in $M_{i,\emptyset}$. Set for convenience $A_t = t(t-2N+J-iK - 3)$. Using Lemma \ref{lemma-e-n}, choose a homogeneous polynomial $g_0$ in $x_{i+1}, x_{i+2},...$ of degree $N+2$ such that ${\mathcal E}_i^{N+1} (g_0) = 0$. Then for $t = 1,...,N$ define $g_t = \frac{1}{A_1\cdots A_t} {\mathcal E}_i^{t} (g_0)$. Note that because of our choice of $N$, $A_t \neq 0$ for $t = 1,...,N$. Then using ${\mathcal E}_i(g_t) = A_{t+1}g_{t+1}$, $t\geq 0$, and ${\mathcal E}_i^{N+1} (g_0) = 0$ one easily verifies that
$$
e_{-i} \left(\sum_{t=0}^N x_i^t g_t\right) = {\mathcal E}_i(g_0) +  \sum_{t=1}^N \left( - A_t x_i^{t-1} g_t + x_i^t {\mathcal E}_i(g_t)\right) =  0.
$$
Since $\sum_{t=0}^N x_i^t g_t$ is zero in $M_{i, \emptyset}$ and $M_i$, we need to multiply this vector by a sufficiently large negative power of $x_i$. To achieve that we next apply the formula (\ref{theta}) on $w = \sum_{t=0}^N x_i^t g_t$ for $u = e_{-i}$, $g = f_i$, and the integer $y = 2N+3 - (J-iK) \geq 1$. More precisely, we use that 
$$
f_i^y e_{-i}  f_i^{-y} (w)= \sum_{j \geq 0} \binom{y}{j}\,
( \ad f_i)^j (e_{-i}) \, f_i^{-j}w.
$$
Note that we used the fact that $N_{i,\emptyset}$ is the $f_i$-localization of $N_{\emptyset}$. However, the right hand side of the above formula is
$$
\left(e_{-i} - \binom{y}{1}(h_0 - i K) f_i^{-1}- 2 \binom{y}{2} f_i^{-1}\right)w = 0
$$
This implies that $e_{-i} (x_i^{-y} w) = 0$. Note that, alternatively, one can directly prove that $e_{-i} \left(\sum_{t=0}^N x_i^ {t-y} g_t \right)$= 0. In any case, $v = \sum_{t=0}^N x_i^ {t-y} g_t$ is an $e_{-i}$-primitive vector in $M_{i, \emptyset}$ and $M_i$. Thus $T_i$ is nonzero. To see that $T_i$ is a submodule it is sufficient to recall that $e_{-i}$ is locally nilpotent on $\mathfrak g$ with respect to the adjoint action.

Now, for arbitrary nonempty set $S = \{j_i,...,j_{\ell} \}$ we need to choose a polynomial $g_0$ of $x_{i+j}, x_{i+j+1},...$, where $j > \max \{j_1,...,j_{\ell} \}$ for all $s$. This is possible since in the proof of Lemma \ref{lemma-e-n} we may choose all $i_1,...,i_{N+2}, i_1',...,i_{n+2}'$ to be bigger than $i + \max \{j_1,...,j_{\ell} \}$. But then we repeat the steps in the proof replacing  $\sum_{t=0}^N x_i^t g_t$ by $y_{j_1}^{-1}\cdots y_{j_{\ell}}^{-1} \sum_{t=0}^N x_i^t g_t$. Hence $v = y_{j_1}^{-1}\cdots y_{j_{\ell}}^{-1} \sum_{t=0}^N x_i^ {t-y} g_t$ is an $e_{-i}$-primitive vector in $M_{i, S}$.
\end{proof}

\begin{corollary} \label{m-i-red}
Let $J-iK$ be an integer. Then both  $M_{i,S}$ ($K \neq 0$) and $M_i$ ($K=0$) are reducible modules.
\end{corollary}

\begin{proof}
Let $K \neq 0$. Assume, on the contrary, that $M_{i,S}$ is irreducible.  Then the space $T_{i, S}$ which  consists of the elements of $M_{i,S}$ on which $e_{-i}$ acts locally nilpotent  forms a submodule of $M_{i,S}$. This is a standard fact and follows from the property that $e_{-i}$ is ad-locally nilpotent on $\mathfrak g$. By Proposition \ref{primitive}(i) $T_{i,S}$ is nontrivial. Hence the action of $e_{-i}$ on all vectors of $M_{i,S} = T_{i,S}$ is locally nilpotent. On the other hand, it is not hard to verify that for a sufficiently large $t$, $e_{-i}^n (x_i^ty_{j_1}^{-1}\cdots y_{j_{\ell}}^{-1}) \neq 0$, for every $n\geq 1$. This leads to a contradiction.  The statement for $M_i$  follows similarly from Proposition \ref{primitive}(ii).
\end{proof}

\subsection{The case $K\neq0$} 
In this subsection we prove the necessary and sufficient conditions for  the modules $M_S$ and $M_{i,S}$ to be irreducible, see Theorem ~\ref{main}(i). We first note that if $K=0$ then both modules are reducible. This follows from the fact that the submodule generated by $y_{j_1}^{-1}\cdots y_{j_{\ell}}^{-1}y_j$ (respectively, $x_i^{-1}y_{j_1}^{-1}\cdots y_{j_{\ell}}^{-1}y_j$) for some $j \notin S$ forms a nontrivial proper submodule of  $M_S$ (respectively, $M_{i,S}$). Next, recall that by Corollary \ref{m-i-red}, $M_{i,S}$ is reducible if $J-iK \in {\mathbb Z}$.  Hence, Theorem~\ref{main}(i)  is an immediate consequence of the following proposition.

\begin{prop}\label{prop-main} Let $K \neq 0$. Then we have the following.

\begin{itemize}
\item[(i)] \label{M'J'} 
\begin{enumerate}

\item[(a)] $U(y_{k_1}^{-1}\dots y_{k_n}^{-1})=M_{S}$.

\item[(b)] Let $v\in M_{S}$ be  nonzero. Then there exists $u\in U$ such that $u(v)=y_{k_1}^{-1}\dots y_{k_n}^{-1}$.
\end{enumerate}

\item[(ii)] \label{prop_prep} Let $J-iK\notin\mathbb Z$. Then
\begin{enumerate}

\item[(a)] $U(x_i^{-1}y_{k_1}^{-1}\dots y_{k_n}^{-1})=M_{i,S}$.

\item[(b)] Let $v\in M_{i,S}$ be nonzero. Then there exists $u\in U$ such that $u(v)=x_i^{-1}y_{k_1}^{-1}\dots y_{k_n}^{-1}$.
\end{enumerate}
\end{itemize}
\end{prop}

\subsubsection{Proof of \propref{prop-main}\rm{(i)}\rm{(a)}}

\begin{lemma}\label{prop4.4}
Let $v=y_{k_1}^{-\gamma_1}\dots y_{k_n}^{-\gamma_n}y_{j_1}^{\beta_1}\ldots y_{j_l}^{\beta_l}\in M_{S}$ and let $s>0$. Then there exists $u_1\in U$ such that $u_1(v)=y_sv$. Moreover, if $s\in\{k_1,\dots,k_n,j_1,\dots,j_l\}$ then there exists $u_2\in U$ such that $u_2(v)=y_s^{-1}v$.
\end{lemma}

\begin{proof}
Note that
\begin{equation*}
h_{-s}(v)=y_sv\text{ and }h_s(v)=c_sy_s^{-1}v,
\end{equation*}
where
\begin{align}\label{csfirst}
c_s:&=\begin{cases}
  -2k_t\gamma_tK,&\text{if }s=k_t,
\\ 2j_t\beta_tK,&\text{if }s=j_t,
\\0,&\text{otherwise.}
\end{cases}
\end{align}

Define $u_1=h_{-s}$ and, if $s\in\{k_1,\dots,k_n,j_1,\dots,j_l\}$, define $u_2=\dfrac{1}{c_s}h_s$.\end{proof}

Now we prove Proposition \ref{prop-main}\rm{(i)}(a). Suppose that
\begin{equation*}
u=f_{i_k}^{\alpha_k}\dots f_{i_1}^{\alpha_1}h_{k_1}^{\gamma_1-1}\dots h_{k_n}^{\gamma_n-1}h_{-j_1}^{\beta_1}\dots h_{-j_l}^{\beta_l},
\end{equation*}
where $\gamma_i\geq 1$ for all $i$. Then following the proof of Lemma \ref{prop4.4},
$u(y_{k_1}^{-1}\dots y_{k_n}^{-1})=Ay_{k_1}^{-\gamma_1}\dots y_{k_n}^{-\gamma_n}y_{j_1}^{\beta_1}\ldots y_{j_l}^{\beta_l}x_{i_1}^{\alpha_1}\ldots x_{i_k}^{\alpha_k}$, for some  $A\in\mathbb C$, $A \neq 0$. Writing an  element of $M_S$ as a sum of monomials, we conclude that $U(y_{k_1}^{-1}\dots y_{k_n}^{-1})=M_S$ .

\subsubsection{Proof of  \propref{prop-main}\rm{(i)}\rm{(b)}}
\begin{lemma} \label{cor-deg-0}
Let $v$ be a nonzero vector in $M_{S}$ such that $\deg_x(v)=0$. There exists $u \in U$ such that $u(v) = y_{k_1}^{-1}\dots y_{k_n}^{-1}$.
\end{lemma}
\begin{proof}
By Lemma \ref{prop4.4}, with consecutive applications of  $h_{-t}$, $t>0$, to $v$, we obtain a nonzero element $v_1 = y_{k_1}^{-1}\dots y_{k_n}^{-1} g({\bf y})$ of $M_{S}$, where $g$ is a polynomial such that $\partial y_{k_i} (g) = 0$. If $\deg_{y}g = 0$, then the needed element $u \in U$ is found. If $\deg_{y}g > 0$ we fix $s \in {\mathbb N}$ such that $\partial y_s (g) \neq 0$ and observe that $h_s v_1 = y_{k_1}^{-1}\dots y_{k_n}^{-1}  \partial y_s (g)$. We complete the proof using induction on  $\deg_{y}g$.
\end{proof}

The following lemma follows by a straightforward computation.
\begin{lemma} \label{e-n}
Let $v\in N_S$. Then for every $j \in {\mathbb Z}$, there is sufficiently large $n>0$ such that 
$$
e_{-n} (v) = y_{n-j} \partial x_j (v) + v'
$$
for some $v' \in N_S$ such that $\partial y_{n-j} (v') = 0$.
\end{lemma}

We prove \propref{prop-main}\rm{(i)}\rm{(b)} by induction on $\deg_x(v)$.
The base  case  $\deg_x(v)=0$ follows from Lemma \ref{cor-deg-0}. Let now $\deg_x(v)\geq 1$ and fix $j \in {\mathbb Z}$ such that $\partial x_j (v) \neq 0$. Then by Lemma \ref{e-n} we find $n\gg 0$ such that $e_{-n}v \neq 0$ in $M_S$. Using the induction hypothesis and the fact that $\deg_x(e_{-n} (v)) <  \deg_x(v) $ we complete the proof.

\subsubsection{Proof of \propref{prop-main}\rm{(ii)}{\rm (a)}}  \label{sub-1}

\begin{lemma}\label{prop4.10}
Let $J-iK \notin {\mathbb Z}$ and  $v=x_i^{-p}y_{k_1}^{-\gamma_1}\dots y_{k_n}^{-\gamma_n}y_{j_1}^{\beta_1}\ldots y_{j_l}^{\beta_l}\in M_{i,S}$. Let $s>0$. Then there exist $u_1,u_2\in U$ such that $u_1(v)=x_i^{-1}v$ and $u_2(v)=y_s v$. Moreover, if $s\in\{k_1,\dots,k_n,j_1,\dots,j_l\}$ then there exists $u_3\in U$ such that $u_3(v)=y_s^{-1}v$.
\end{lemma}

\begin{proof}
Set
\begin{equation*}
u_1=\dfrac{-1}{p(p+1+J-iK)}e_{-i} \text{ and }
u_2=h_{-s}+\dfrac{2}{p+1+J-iK}f_{i-s}e_{-i}.
\end{equation*}
Also  for $s\in\{k_1,\dots,k_n,j_1,\dots,j_l\}$ set
\begin{align*}
 u_3=-\dfrac{1}{c_s}\Big(h_s+\dfrac{2}{p+1+J-iK}f_{i+s}e_{-i}\Big),
\end{align*}
where $c_s\neq0$ is defined in \eqnref{csfirst}.  Then we have
$ u_1(v)=x_i^{-1}v, u_2(v)=y_sv$, and $u_3(v)=y_s^{-1}v$.
\end{proof}
\begin{remark} \label{non-int}
The last lemma can be easily generalized to the case when $p$ is any complex number such that $p + J - iK$ is not an integer. This generalization will be used when twisted localizations of $M_{i,S}$ are considered (see Theorem \ref{main5}).
\end{remark}

Applying similar argument as in the proof of Proposition \ref{prop-main}\rm{(i)} and using consecutive applications of Lemma \ref{prop4.10}
 one can easily show that $U(x_i^{-1}y_{k_1}^{-1}\dots y_{k_n}^{-1})=M_{i,S}$. Hence Proposition \ref{prop-main}\rm{(ii)}(a) follows.

\subsubsection{Proof of Proposition  \ref{prop-main}\rm{(ii)}\rm{(b)}} \label{sub-2}

\begin{lemma}\label{deg0}
Let $J-iK \notin {\mathbb Z}$ and $v$ be nonzero vector in $M_{i,S}$ such that $\deg_{x^+}(v)=0$. Then there exists $u\in U$ such that $u(v)=x_i^{-1}y_{k_1}^{-1}\dots y_{k_n}^{-1}$.
\end{lemma}
\begin{proof}
With consecutive applications of Lemma \ref{prop4.10} we can find $u_1 \in U$ so that $u_1(v) = x_i^{-1}y_{k_1}^{-1}\dots y_{k_n}^{-1}g ({\bf y})$ for some nonzero polynomial $g$ such that $\partial y_{k_1} (g) =\cdots =\partial y_{k_n} (g)=0$. If $\deg_{y} g = 0$, then the needed element $u$ is found. Assume now $\deg_{y} g \geq 1$ and let $j > 0$ be such that $\partial y_j (g) \neq 0$. Then using that $u_2 (x_i^{-1}y_{k_1}^{-1}\dots y_{k_n}^{-1}g) = 2jKx_i^{-1}y_{k_1}^{-1}\dots y_{k_n}^{-1}\partial y_j (g)$ for $u_2 = h_j+\dfrac{2}{2+J-iK}f_{i+j}e_{-i}$ we complete the proof by induction on $\deg_{y} g$.
\end{proof}

\begin{definition}\label{defi} For $p \in {\mathbb C}$ and $\alpha \in {\mathbb N}$, define
$$
A_{\alpha,p}=-p(p+1-2\alpha+J-iK).
$$
\end{definition}
The above definition is motivated by the following.

\begin{lemma}\label{degi+}
Assume that $Z=y_{k_1}^{-\gamma_1}\dots y_{k_n}^{-\gamma_n}y_{j_1}^{\beta_1}\ldots y_{j_l}^{\beta_l}x_{i_1}^{\alpha_1}\ldots x_{i_k}^{\alpha_k}$, where $p, \gamma_1,\ldots,\gamma_n,$ $\beta_1,\ldots,\beta_l,$ $\alpha_1,\ldots,\alpha_k\in\mathbb N$. Then
\begin{equation*}
e_{-i}(x_i^{-p}Z)=A_{\alpha,p}x_i^{-p-1}Z+x_i^{-p}Z'+x_i^{-p+1}Z''
\end{equation*}
in $M_{i,S}$, where $\alpha=\displaystyle\sum_{s=1}^k\alpha_s = \deg_{x^+} (x_i^{-p}Z)$, and $Z'$ and $Z''$ are  $\deg_{x^+}$-homogeneous polynomials which are either zero in $M_{i,S}$ or $\deg_{x^+}(x_i^{-p+1}Z'')=\deg_{x^+}(x_i^{-p}Z')-1=\alpha-2$. In particular, $f_ie_{-i}(x_i^{-1}Z)=A_{\alpha,1}x_i^{-1}Z$.
\end{lemma}

\begin{proof} To prove the statement we write $e_{-i}$ (or, more precisely, $\Phi(e_{-i})$) as a sum of three $\deg_{x^+}$-homogeneous operators $e_{-i} = e_{-i,0} + e_{-i,1} + x_ie_{-i,2}$, where
\begin{eqnarray*}
e_{-i,0}& =& -x_i\partial x_i\partial x_i-2\Big(\displaystyle\sum_{s\in\mathbb Z\setminus\{i\}}x_s\partial x_s\Big)\partial x_i+(J-iK)\partial x_i,\\
e_{-i,1}& =& -\displaystyle\sum_{s_1,s_2\in\mathbb Z\setminus\{i\}\atop{s_1+s_2\neq 2i}}x_{s_1+s_2-i}\partial x_{s_1}\partial x_{s_2}+\displaystyle\sum_{s>0}y_s\partial x_{i-s}+2K\displaystyle\sum_{t>0}t\partial y_t\partial x_{t+i},\\
e_{-i,2}& =& -\displaystyle\sum_{s_1,s_2\in\mathbb Z\setminus\{i\}\atop{s_1+s_2=2i}}\partial x_{s_1}\partial x_{s_2}=-\displaystyle\sum_{s\in\Z\setminus\{i\}}\partial x_{2i-s}\partial x_s.
\end{eqnarray*}
Then we take $Z'=e_{-i,1}Z\text{ and }Z''=e_{-i,2}Z$ and verify that $e_{-i,0}(x_i^{-p}Z) = A_{\alpha,p}x_i^{-p-1}Z$.  \end{proof}

\begin{lemma} \label{e-n-i}
Let $J - iK\notin {\mathbb Z}$. Let $v = x_i^{-1}h$, where  $h = h({\bf x}, {\bf y})$ is a $\deg_{x^+}$-homogeneous polynomial such that $\deg_{x^+} h = \alpha$ and $\partial x_i (h) = 0$. Then there is $n > i$ such that $w = \left( e_{-i}f_i - A_{\alpha, 1}\right) \left( e_{-i}f_i - A_{\alpha+1, 2}\right) e_{-n} (v)$ is nonzero in $M_{i,S}$ and $ \deg_{x^{+}} w< \deg_{x^{+}}v$.

\end{lemma}
\begin{proof} Let us first consider the case $ h({\bf x}, {\bf y})  =x_{i_1}^{\alpha_1}\dots x_{i_k}^{\alpha_k} g({\bf y})$, i.e. $h$ is a monomial in ${\bf x}$ and polynomial in ${\bf y}$. Then one can verify that for $n> \max \{i,i_1,...,i_k \}$,
\begin{eqnarray*}e_{-n}(x_i^{-1}h)& =&  -2 x_i^{-3}x_{2i-n}h + 2 x_i^{-2}\left( \sum_{j=1}^k\alpha_jx_{i+i_j-n}x_{i_j}^{-1}h-y_{n-i}h\right)\\
&& + x_i^{-1} \left( \sum_{j=1}^k\alpha_jx_{i_j}^{-1}y_{n-i_j}h-\sum_{j,t=1}^k\alpha_{i_{j},i_{t}}x_{i_{j}}^{-1}x_{i_{t}}^{-1}x_{i_{j}+i_{t}-n}h \right)
\end{eqnarray*}
where $\alpha_{i_{j},i_{t}}=\alpha_{i_{j}} (\alpha_{i_{t}} - \delta_{j,t})$. Set for simplicity  
$$
\widetilde{e}_n = \left( e_{-i}f_i - A_{\alpha, 1}\right) \left( e_{-i}f_i - A_{\alpha+1, 2}\right) e_{-n}.
$$ Using the above formula one can compute that for every $j = 1,...,k$ and sufficiently large $n$ we have
\begin{equation*}
w = \widetilde{e}_n v = \alpha_jC_{\alpha} x_i^{-1}x_{i_j}^{-1}y_{n-i_j}h + g_{i_j},
\end{equation*}
where 
$$
C_{\alpha} = A_{\alpha,1}A_{\alpha+1,2}-2A_{\alpha+1,2}+4 = 2 (2-2\alpha+J-iK)(3-2\alpha+J-iK)
$$
and $\partial y_{n-i_j} g_{i_j} = 0$. On the other hand $\partial y_{n-r} (\widetilde{e}_n (x_i^{-1}h)) = 0$ if $r \neq i_j$, $j=1,...,k$.  Therefore for every $i'\neq i$,
\begin{equation}\label{v-w}
\partial y_{n-i'} (\widetilde{e}_n (x_i^{-1}h)  - C_{\alpha} x_i^{-1}y_{n-i'}\partial x_{i'} h) = 0.
\end{equation}
Since $J-iK \notin\mathbb Z$ we have  $C_{\alpha} \neq 0$. Also $\partial x_{i_{j}} h \neq 0$ for all $j = 1, \ldots, k$.  We conclude 
that $w \neq 0$.

Consider now an arbitrary polynomial $h({\bf x}, {\bf y})$ and fix $i'$ such that $\partial x_{i'} h \neq 0$. Let $h = h_1+\cdots + h_t$, where each $h_i$ is a monomial in ${\bf x}$. Applying (\ref{v-w}) to each $h_s$, and then taking the sum over $s = 1, \ldots, t$ one easily shows that 
$$
w = C_{\alpha}  x_i^{-1} y_{n-i'} \partial x_{i'}(h) + h'
$$
for some $ h'$ such that  $\partial y_{n-i'} (h') = 0$. Hence $w\neq 0$. The inequality  $ \deg_{x^{+}} w< \deg_{x^{+}}v$ follows from Lemma \ref{degi+}.
\end{proof}

\begin{lemma}\label{h0D}
Let $v \in M_{i,S}$. Then there is $u \in U$ such that $u(v)$ is a nonzero $\deg_x$-homogeneous element in $M_{i,S}$.
\end{lemma}
\begin{proof}
The argument is standard but for a sake of completeness we provide a short proof. Note that $w \in M_{i,S}$ is $\deg_x$-homogeneous if and only if $h_0 (w) = \lambda w$ for some $\lambda \in {\mathbb C}$. Present $v$ as a sum $v = w_1+\dots+w_{t}$ of weight vectors, and let $h_0 (w_i) = \lambda_i w_i$, where $\lambda_1, \ldots, \lambda_t$ are pairwise distinct. Then for $k=0, \dots, t-1$ we have 
$$
h_0^k(v)=\lambda_1^kw_1+\dots+\lambda_t^k w_{ t}.
$$
Solving this linear system in $w_1, \ldots, w_t$ and using the fact that the corresponding Vandermonde determinant is nonzero  we find $u_i\in U$ such that $u_i(v)=w_i$ for all $i$.
\end{proof}

Now we are ready to prove Proposition  \ref{prop-main}\rm{(ii)}{\rm (b)}. 
After replacing $v$ by $f_i^N(v)$ for some $N\geq 0$ if necessary, we may assume that $v = x_i^{-1} g({\bf x}, {\bf y})$ for a polynomial $g$ such that $\partial x_i (g) = 0$. Furthermore, using  Lemma \ref{h0D}, we may assume that $g$ is $\deg_{x}$-homogeneous, hence $\deg_{x^+}$-homogeneous. We continue by induction on $\deg_{x^+}(v) = \deg_{x^+}(g) $.  The base case  $\deg_x(v)=0$ follows from Lemma \ref{deg0}. Let now $\deg_x(v) = \alpha \geq 1$ and fix $j \in {\mathbb Z}$ such that $\partial x_j (v) \neq 0$. Then by Lemma \ref{e-n-i} we find $n>>0$ such that $\widetilde{e}_{-n}v \neq 0$ in $M_{i,S}$, where $\widetilde{e}_n  = \left( e_{-i}f_i - A_{\alpha, 1}\right) \left( e_{-i}f_i - A_{\alpha+1, 2}\right) e_{-n}$. Using the induction hypothesis and the fact that $\deg_x(\widetilde{e}_{-n} (v)) <  \deg_x(v) $ we complete the proof.

\subsection{The case $K=0$} 
In this section we prove the necessary an sufficient condition for $M_i = \Phi'(i)$ to be irreducible, see Theorem \ref{main}(ii). By Corollary \ref{m-i-red}, we know that $M_i$ is reducible when $J \in {\mathbb Z}$. So, to prove Theorem \ref{main}(ii), it is sufficient to prove the following proposition. 

\begin{prop} \label{prop-main-2}Let $K= 0$ and $J \notin {\mathbb Z}$. Then we have the following.
\begin{enumerate}
\item[(i)] $U(x_i^{-1})=M_{i}$.

\item[(ii)] Let $v\in M_{i,S}$ be nonzero. Then there exists $u\in U$ such that $u(v)=x_i^{-1}$.
\end{enumerate}
\end{prop}

We start with the following lemma which implies Proposition \ref{prop-main-2}\rm{(i)}.
\begin{lemma}
Let $v\in M_i$ be a nonzero element such that $\deg_{x^+}(v)=0$. Then there exists $u_1\in U$ such that $u_1(v)=x_i^{-1}$. Moreover, $U(x_i^{-1})=M_i$.
\end{lemma}

\begin{proof}Fix $p\in\mathbb N$. For $v=x_i^{-p}$ define $u_1=f_i^{p-1}$.  We have
\begin{equation*}
\dfrac{-1}{p(p+1+J)}e_{-i}(x_i^{-p})=x_i^{-p-1}.
\end{equation*}
The above  identity together with $\Big(\Pi_{s=1}^kf_{i_s}^{\alpha_s}\Big)(x_i^{-p})=x_i^{-p}x_{i_1}^{\alpha_1}\dots x_{i_n}^{\alpha_n}$ imply $U(x_i^{-1})=M_i$.
\end{proof}

Proposition \ref{prop-main-2}\rm{(ii)} follows from Lemma \ref{h0D} (applied for $S = \emptyset$) and the following lemma.
\begin{lemma}
For $v\in M_i$ such that $\deg_{x^+}(v)>0$ there exists $u\in U$ with $u(v)=v'\neq0$ and $\deg_{x^+}(v')< \deg_{x^+}(v)$.
\end{lemma}
\begin{proof} After applying a power of $f_i$ to $v$ if necessary, we first assume $v=x_i^{-1}x_{i_1}^{\alpha_1}\dots x_{i_k}^{\alpha_k}$.  For every $s \in {\mathbb Z}$ we have $e_{-s} (v) = v_s + \sum_{j=1}^k\delta_{s,i_j}\alpha_jJx_{s}^{-1} v$, where 
\begin{align*}
v_s&=-2x_i^{-3}x_{i_1}^{\alpha_1}\dots x_{i_k}^{\alpha_k}x_{2i-s}+2 x_i^{-2}\Big(\displaystyle\sum_{m=1}^k\alpha_mx_{i_m}^{-1}x_{i_1}^{\alpha_1}\dots x_{i_k}^{\alpha_k}x_{-s+i+i_m}\Big)\\
&+x_i^{-1} \left( \displaystyle\sum_{m_1,m_2=1}^k\alpha_{m_1}(\alpha_{m_2}-\delta_{m_1,m_2})x_{-s+i_{m_1}+i_{m_2}}x_{i_{m_1}}^{-1}x_{i_{m_2}}^{-1} x_{i_1}^{\alpha_1}\dots x_{i_k}^{\alpha_k} \right)
\end{align*}

Following the reasoning of the proof of Lemma \ref{e-n-i}, set $\widetilde{e}_s = \left( e_{-i}f_i - A_{\alpha, 1}\right) \left( e_{-i}f_i - A_{\alpha+1, 2}\right) e_{-s}$. Also, let  $\widetilde{v}_s = \left( e_{-i}f_i - A_{\alpha, 1}\right) \left( e_{-i}f_i - A_{\alpha+1, 2}\right) v_{s}$. Assume that $\widetilde{e}_s (v) =0$ for all integer $s$, $s \notin \{i_1,...,i_k\}$. But then $v_s = 0$ for all such $s$. Since the coefficients of $v_s$ do not depend on $s$, we have $v_s=0$ for all $s \in {\mathbb Z}$. Hence $\widetilde{e}_{i_1}(v) = A_{\alpha,1}A_{\alpha+1,2}\alpha_1J v$ is nonzero because  $J \notin {\mathbb Z}$ and $A_{\alpha,1}$ and $A_{\alpha+1,2}$ are nonzero.  Since $\widetilde{e}_{i_1}(v)$ has  strictly smaller $x^{+}$-degree than  $\deg {x^{+}} (v)$, we proved the statement for the monomial $v$.

For arbitrary polynomial $v = v_1+ \cdots + v_t$  with linearly independent $v_1,....,v_t$, we apply the same reasoning. First we assume that all monomials have $x_i$-degree $-1$. Then $v_s = (v_1)_s+\cdots + (v_t)_s$ has coefficients that do not depend on $s$. If we assume $\widetilde{e}_s (v) =0$ for all integer $s$, $s \notin \{i_1,...,i_k\}$, then we will have 
$$
\widetilde{e}_{i_1}(v) = A_{\alpha,1}A_{\alpha+1,2} J ( \alpha_1 (v_1)  v_1+\dots + \alpha_1 (v_t) v_t),
$$
where $\alpha_1 (v_j)$ stands for the $x_{i_1}$-degree of $v_j$. Now using that the monomials $v_1,...,v_t$ are linearly independent and that $A_{\alpha,1}A_{\alpha+1,2} J \neq 0$ we complete the proof. \end{proof}

\section{Twisted localization of imaginary Verma modules} \label{sec-im-ver}

 We apply the localization functors $\cD_g$ and $\cD_g^z$ defined in \S \ref{loc} and \S \ref{gencon} to imaginary Verma modules when $g = f_i$ (recall  that $\Phi(f_i) = x_i$). Denote for simplicity $\cD_i = \cD_{f_i}$ and $\cD_i^z = \cD_{f_i}^z$.

\subsection{The case of nonzero central charge}

Let $\lambda\in \mathfrak{h}^*$ be such that $\lambda(c)\neq 0$. We also assume for simplicity that $\lambda(d) =0$, see Remark \ref{rem-d}(i). Recall that by Theorem \ref{th-im-ver},  we can identify  the imaginary Verma module $M(\lambda)$ of highest weight $\lambda$ with $\C[\bf{x},\bf{y}]$ by setting $J = \lambda(h_0)$ and $K = \lambda(c)$. To describe the twisted localization of $M(\lambda)$ we introduce the space of twisted polynomials: for  $ z\in\mathbb C$, the space $x_i^{ z} \C[x_i^{{\pm1}},\bf{\widehat{x}_i},\bf{y}]$ is spanned by $x_i^{ z} p$, $p \in \C[x_i^{{\pm1}},\bf{\widehat{x}_i},\bf{y}]$. The action of ${\mathcal A} ({\bf x}, {\bf y})$ (and hence of $U = U(\widehat{\mathfrak{sl}}_2)$) on $x_i^{ z} \C[x_i^{{\pm1}},\bf{\widehat{x}_i},\bf{y}]$ is obvious. The next lemma  is straightforward.
\begin{lemma} \label{lem-loc}
(i) $\cD_i M(\lambda) \simeq \C[x_i^{{\pm1}},\bf{\widehat{x}_i},\bf{y}]$, and hence $\cD_i M(\lambda)/M(\lambda) \simeq M_{i, \emptyset}$.

(ii) $\cD_i^{z} M(\lambda)\simeq x_i^{ z}\C[x_i^{{\pm1}},\bf{\widehat{x}_i},\bf{y}]$.
\end{lemma}

\begin{remark}
Note that $x_i^{ z}\C[x_i^{{\pm1}},\bf{\widehat{x}_i},\bf{y}]$ is both a $U$-module and a $\cD_iU$-module. We will make use of both considerations. For example, we have that for $u \in U$, $k \in \Z$, and $p \in x_i^{ z}\C[x_i^{{\pm1}},\bf{\widehat{x}_i},\bf{y}]$, $(uf_i^{k})p = u(x_i^kp)$.
\end{remark}

For any root $\gamma\in \Delta^{\mathsf {re}}$ choose  a nonzero element
$X_{\gamma}$  in the $\gamma$-root space. Denote by $\mathcal W$  the
category of weight modules of $\widehat{\mathfrak{sl}}_2$ and by $\mathcal W(\alpha)$ its full
subcategory consisting of modules such that $X_\alpha$ is locally
nilpotent and, for any $\gamma\in\Delta^{\mathsf {re}},\
\gamma\neq\alpha$, $X_\gamma$ acts injectively. More generally, for a subset $A$ of $\Delta^{\mathsf {re}}$, denote by $\mathcal W(A)$ the category of weight modules $M$  on which $X_\gamma$ is locally nilpotent for $\gamma \in A$ and injective for all other $\gamma\in\Delta^{\mathsf {re}}$. Note that $A = \emptyset$ is allowed, and the modules in $\mathcal W(\emptyset)$ are  torsion free because all real root elements act injectively.
As mentioned in the beginning of Section \ref{sec-prel}, a torsion free module  is dense, but not vice versa.

Our second main result is the following theorem.

\begin{thm}\label{main5} Let $i \in \Z$, $ z \in \C$, and $\lambda \in {\mathfrak h}^*$ be such that $\lambda (c) = K\neq 0$ and $\lambda(h_0) = J$.
\begin{itemize}
\item[(i)] Let  $ z\in\mathbb Z$. Then $\cD_i^ z M(\lambda) \simeq \cD_i M(\lambda)$ and  $\cD_i M(\lambda)/ M(\lambda)$ is  irreducible if and only if $J-iK \notin \mathbb Z$. In the latter case, $\cD_i M(\lambda)/ M(\lambda)$ is a dense module in $\mathcal W(- \beta + i \delta)$ all of whose weight multiplicities are infinite.

\item[(ii)] Let $ z \notin \Z$. Then $\cD_i^ z M(\lambda)$ is irreducible if and only if $z -J + iK \notin {\mathbb Z}$. In the latter case, $\cD_i^ z M(\lambda)$ is a dense torsion free module all of whose weight multiplicities are infinite. 

\item[(iii)] $\cD_{i_1}^{ z_1} M(\lambda)\simeq \cD_{i_2}^{ z_2} M(\lambda)$ if and only if $ z_1- z_2\in\mathbb Z$ and $ i_1=i_2$.
\end{itemize}
\end{thm}
\begin{proof}
(i) Recall that by Lemma \ref{lem-loc}(i), $\cD_i M(\lambda)/ M(\lambda) \simeq  M_{i,\emptyset}$. The necessary and sufficient condition for the simplicity of  $\cD_i M(\lambda)/ M(\lambda)$  follows from Theorem \ref{main}(i). The density and the infiniteness of the multiplicities follow from the vector space isomorphism 
$$\cD_i M(\lambda)/ M(\lambda) \simeq x_i^{-1}\mathbb C[x_i^{-1},\widehat{{\bf x_i}},{\bf y}].$$
 In order to check which root elements act locally nilpotently on the irreducible $U$-module $\cD_i M(\lambda)/ M(\lambda)$, it is sufficient to check which $e_{\alpha}$ and $f_{\alpha}$ act locally nilpotent on $x_i^{-1}$. The latter is an easy computation and we see that there is exactly one such root element, namely,  $e_{\alpha}$ for $\alpha = - \beta + i \delta$.

(ii) We first note that if  $N= z -J + iK$ is in ${\mathbb Z}$, then $e_{-i} (x_i^{z - N + 1}g (y)) = 0$ for any element $g$ in $\mathbb C[{\bf y}]$ (i.e. $\deg_x g = 0$). In particular, the $e_{-i}$--locally nilpotent elements in $\cD_i^z M(\lambda)$ form a nontrivial submodule, and hence  $\cD_i^z M(\lambda)$ is not irreducible. Assume now that $z -J + iK \notin {\mathbb Z}$. We first show that $U(x_i^{z}) = \cD_i^z M(\lambda)$. This follows from multiple applications of the nonintegral version of Lemma \ref{prop4.10} (see Remark \ref{non-int}). To show that $\cD_i^z M(\lambda)$ is irreducible, it is sufficient to verify that for every $v = x_i^{z} p$, $p \in \C[x_i^{{\pm1}},\bf{\widehat{x}_i},\bf{y}]$, there is $u \in U$ such that $u(v) = x_i^{z}$. We prove this statement  in three steps - we first reduce the statement to the case when $p$ is $\deg_x$-homogeneous in ${\mathbb C} [{\bf x}, {\bf y}]$, second we find such $u$ in the case 
\begin{equation} \label{f-e}
f_ie_{-i} v = Av \mbox{ for some }A \in {\mathbb C},
\end{equation}
and, lastly, we find $u' \in U$ such that $v'=u'(v)$ has the property (\ref{f-e}).

The first step follows from Lemma \ref{h0D} and, if necessary, applying a change of parameter $ z \mapsto z+N$ for a sufficiently large integer $N$. Assume now that $v$ has the property (\ref{f-e}). Without loss of generality, we may assume that $p = \sum_{j = 0}^t x_{i}^jp_j$, where $\partial x_i (p_j) = 0$, $\deg_{x}{p_0} = \deg_{x}{p} =   \alpha$. Using direct computations we find that $A = A_{\alpha, - z}$. Hence $e_{-i}(v) = A_{\alpha, -z} f_i^{-1} v$. Then using induction and the fact that $h_0 (v) = (-2\alpha - 2z + J)v$ and $f_ie_{-i}f_i^{-1} = e_{-i} - (h_0 -iK)f_i^{-1}$, we prove that for every $k \geq 1$,
\begin{equation} \label{e-i-k}
e_{-i}^k(v) = A_{\alpha, -z} A_{\alpha, -z+1}\cdots  A_{\alpha, -z+k-1}f_i^{-k} v.
\end{equation}
By the simplicity of the $U$-module ${\mathbb C} [{\bf x}, {\bf y}]$, we know that there is $u_1 \in U$ such that $u_1 (p) =1$ in ${\mathbb C} [{\bf x}, {\bf y}]$. Recall the definition of $\Theta_z$ in \S \ref{gencon}. Fix $K\geq 0$ such that $\Theta_z (u_1) = u_2 f_i^{-K}$ for some $u_2 \in U$. Then 
using Lemma \ref{lem-conj} and (\ref{e-i-k}), we have 
$$x_i^z = x_i^z u_1(p) = \Theta_z(u_1) (x_i^z p)= u_2 f_i^{-K} (x_i^z p) = \frac{1}{C}u_2e_{-i}^k
 (x_i^z p),$$
where $C = A_{\alpha, -z} A_{\alpha, -z+1}\cdots  A_{\alpha, -z+k-1}$ is nonzero because $z \notin {\mathbb Z}$ and $z-J+iK \notin {\mathbb Z}$. This completes the second step. For the last step take $v' = (f_ie_{-i} - A_{\alpha, - z})v$,  $\alpha = \deg_x p$, and assume $v' = 0$. We note that $v' = x_i^z p'$, where $\alpha' = \deg_x p' < \deg_x p$. If $v'$ does not satisfy (\ref{f-e}) we apply $f_ie_{-i} - A_{\alpha', - z}$ and so on until we obtain an element $v'' = x_i^{z+n} p''$, $n \in {\mathbb Z}$, with the property $\deg_x p'' = 0$. But then $(f_ie_{-i} - A_{-z-n, 0})v'' = 0$,  and hence $v''$ satisfies (\ref{f-e}).

(iii) The ``if'' part is trivial. Assume $\cD_{i_1}^{ z_1} M(\lambda)\simeq \cD_{i_2}^{ z_2} M(\lambda)$. From Lemma \ref{sup-lemma}, we have that $\lambda + {\mathbb Z} (-\beta + i_1 \delta) + z_1 {\mathbb Z} = \lambda + {\mathbb Z} (-\beta + i_2 \delta) + z_2 {\mathbb Z}$. This identity easily implies $i_1 = i_2$ and $z_1 - z_2 \in {\mathbb Z}$. \end{proof}

\begin{remark}
Theorem \ref{main5} can be  stated  in more general setting if we consider ${\mathcal A} ({\bf x}, {\bf y})$-modules. Namely, in the statement of the theorem we may replace $\cD_i M(\lambda)/ M(\lambda)$ and $\cD_i^ z M(\lambda)$  by the ${\mathcal A} ({\bf x}, {\bf y})$-modules $\cD_i M_S/M_S \simeq M_{i, S}$ and $\cD_i^ z M_S$, respectively (in the case $S = \emptyset$ the corresponding modules coincide). Note that  the functor $\cD_i$ on ${\mathcal A} ({\bf x}, {\bf y})$-modules corresponds to the localization with respect to the multiplicative subset $\{ x_i^n \; | \; n \in {\mathbb Z}_{\geq 0}\}$ of  ${\mathcal A} ({\bf x}, {\bf y})$. 
\end{remark}

\subsection{The case of zero central charge}

Let $\lambda\in \mathfrak{h}^*$. In this subsection we discuss the case of  trivial central
charge, that is $\lambda(c)=0$. In this case the imaginary Verma module $M(\lambda)$ is
reducible and has a proper submodule $\mathfrak{M}$ generated by $M_{\lambda - k \delta}$, $k>0$.  The quotient $M'(\lambda)=M(\lambda)/\mathfrak{M}$ is irreducible if and
only if $\lambda (h_0) \neq 0$ (see Proposition \ref{prop-imag}). Recall that by Theorem \ref{th-im-ver-quot},  $M'(\lambda)$ is isomorphic to ${\mathbb C} [{\bf x}]$ through $\Phi'$. We have the following result for the twisted localization of $M'(\lambda)$, the proof of which is identical to the proof of Theorem \ref{main5} with the only difference that the polynomials considered in the proof are in $\C[x_i^{{\pm1}},\bf{\widehat{x}_i}]$ instead of $\C[x_i^{{\pm1}},\bf{\widehat{x}_i},\bf{y}]$.

\begin{thm}\label{main6} Let $i \in \Z$, $ z \in \C$, and $\lambda \in {\mathfrak h}^*$ be such that $\lambda (c) = 0$ and $\lambda(h_0) = J$.
\begin{itemize}
\item[(i)] Let  $ z\in\mathbb Z$. Then $\cD_i^ z M'(\lambda) \simeq \cD_i M'(\lambda)$ and  $\cD_i M'(\lambda)/ M'(\lambda)$ is  irreducible if and only if $J \notin \mathbb Z$. In the latter case, $\cD_i M'(\lambda)/ M'(\lambda)$ is a dense module in $\mathcal W(- \beta + i \delta)$ all of whose weight multiplicities are infinite.

\item[(ii)] Let $ z \notin \Z$. Then $\cD_i^ z M(\lambda)$ is irreducible if and only if $z - J \notin {\mathbb Z}$. In the latter case, $\cD_i^ z M'(\lambda)$ is a dense torsion free module all of whose weight multiplicities are infinite. 

\item[(iii)] $\cD_{i_1}^{z_1} M'(\lambda)\simeq \cD_{i_2}^{z_2} M'(\lambda)$ if and only if $z_1- z_2\in\mathbb Z $ and $i_1=i_2$.
\end{itemize}
\end{thm}

\begin{remark}
Theorems \ref{main5}  and \ref{main6}  provide a tool to explicitly construct  new irreducible  weight dense modules for  $\widehat{\mathfrak{sl}}_2$. All such modules have  infinite weight multiplicities. 
If $\mathfrak G$ is an arbitrary affine Lie algebra then one can consider a parabolic subalgebra of $\mathfrak G$ whose Levi factor is isomorphic to  $\widehat{\mathfrak{sl}}_2+ \mathfrak H$, where $\mathfrak H$ is a Cartan subalgebra of $\mathfrak G$ containing $\mathfrak h$. Then defining an ${\mathfrak H}$-module structure on $M_S, M_{i,S}, M_i, \cD_i^ z M_S, \cD_i^ z  \mathbb C [{\bf x}]$ (K=0 for the latter two), one can apply the induction functor on these modules. By taking the unique irreducible quotients of the induced modules, we will obtain new irreducible modules with infinite weight multiplicities for any affine Lie algebra.
\end{remark}

\section{Acknowledgment}
 The first author was supported in part by
the CNPq grant (301320/2013-6) and by the FAPESP grant
(2014/09310-5).  The second author was supported by FAPESP grant (2011/21621-8) and NSA grant H98230-13-1-0245. The third author was supported by FAPESP grant
(2012/02459-8). This work was completed during the second author's stay at IH\'ES and he would like to thank the institute for the warm hospitality and  excellent working conditions.

\end{document}